\documentclass[11pt]{amsart}

\usepackage{geometry}
\usepackage{todonotes}

\usepackage{amsfonts, amssymb, amscd}
\numberwithin{equation}{section}

\usepackage[symbol]{footmisc}

\usepackage{bm}
\usepackage{verbatim}
\usepackage{mathrsfs}
\usepackage{graphicx}
\usepackage{tikz-cd}
\usepackage{subcaption}
\usepackage{listings}
\usepackage{subfiles}
\usepackage[toc,page]{appendix}
\usepackage{mathtools}
\usepackage{comment}
\usepackage{enumerate}
\usepackage{enumitem}

\usepackage[all,color]{xy}

\usepackage{graphicx}
\graphicspath{{images/}}

\usepackage{appendix}
\usepackage{hyperref}
\newcommand{\bR}{\mathbb{R}}
\newcommand{\bQ}{\mathbb{Q}}

\newcommand{\cO}{\mathcal{O}}

\newcommand{\bZ}{\mathbb{Z}}

\newcommand{\Qq}{\mathbb{Q}}

\newcommand{\Rr}{\mathbb{R}}

\newcommand{\Zz}{\mathbb{Z}}

\newcommand{\Exc}{\operatorname{Exc}}

\newcommand{\coker}{\operatorname{coker}}

\newcommand{\Spec}{\operatorname{Spec}}
\newcommand{\Supp}{\operatorname{Supp}}

\newcommand{\mult}{\operatorname{mult}}

\newcommand{\lf}{\lfloor}
\newcommand{\rf}{\rfloor}

\newcommand{\Ff}{\mathcal{F}}

\newcommand{\Oo}{\mathcal{O}}

\newcommand{\Xx}{\mathcal{X}}

\newcommand{\Ex}{\mathrm{Ex}}

\newtheorem{thm}{Theorem}[section]

\newtheorem{cor}[thm]{Corollary}
\newtheorem{lem}[thm]{Lemma}
\newtheorem{prop}[thm]{Proposition}

\newtheorem{claim}[thm]{Claim}

\theoremstyle{definition}
\newtheorem{defn}[thm]{Definition}

\theoremstyle{definition}
\newtheorem{rem}[thm]{Remark}

\theoremstyle{definition}

\begin{document}

	\title[On the termination of the MMP for semi-stable fourfolds]{On the termination of the MMP for semi-stable fourfolds in mixed characteristic}
	\author{Lingyao Xie and Qingyuan Xue}

	\address{Department of Mathematics, The University of Utah, Salt Lake City, UT 84112, USA}
	\email{lingyao@math.utah.edu}

	\address{Department of Mathematics, The University of Utah, Salt Lake City, UT 84112, USA}
	\email{xue@math.utah.edu}
	

	\begin{abstract}
		We improve on the result of Hacon and Witaszek by showing that the MMP for semi-stable fourfolds in mixed characteristic terminates in several new situations.
		In particular, we show the validity of the MMP for strictly semi-stable fourfolds over excellent Dedekind schemes globally when the residue fields are perfect and have characteristics $p>5$.
	\end{abstract}
	
	\maketitle
	
	\tableofcontents

	\section{Introduction}
	
	Recently there has been much progress in developing the Minimal Model Program (MMP) in positive and mixed characteristic. For surfaces the theory is classical, and the MMP for excellent surfaces was proved in \cite{Tan18}. For threefolds over a perfect field $k$ with char $p>3$, most of the important results in the MMP were established in \cite{HX15,CTX15,Bir16,BW17,GNT19,HW21,HW19}; see also \cite{Wal18,HNT20}. When the base field $k$ is imperfect (and has characteristic $>5$), some results, for example the existence of minimal models, were proved in \cite{DW19}. As for mixed characteristic, \cite{BMP+20} established the MMP for arithmetic threefolds whose residue characteristics are $p>5$ (see \cite{XX22} for the case $p>3$), and \cite{TY20} established the MMP for strictly semi-stable threefolds over excellent Dedekind schemes and some birational cases. The MMP for varieties with dimension greater than 3 is much harder, but \cite{HW20} proved the validity of some special MMP for fourfolds in positive and mixed characteristic.
	
	Our main purpose in this article is to show the validity of the MMP for strictly semi-stable fourfolds in mixed characteristic (and some partial results in positive characteristic). To achieve this, we generalize the result of \cite[Theorem 1.2]{HW20} by proving the termination of MMP in the non-effective case and the existence of Mori fiber spaces.
	
	Throughout this paper, we assume that log resolutions of all log pairs with the underlying varieties being birational to $X$ as below exist and are given by a sequence of blow-ups along the non-snc locus.
	
	\begin{thm}\label{thm: main thm mixed}
		Let $(X, \Delta)$ be a four-dimensional $\bQ$-factorial dlt pair projective over a discrete valuation ring $R$, where $R$ is of mixed characteristic and has perfect residue field with characteristic $p > 5$. Let $s \in \Spec R$ be the special point and let $\phi: X \to \Spec R$ be the natural morphism.
		
		Suppose that $\Supp(\phi^{-1}(s)) \subseteq \lfloor \Delta \rfloor$. Then we can run a $(K_X+\Delta)$-MMP with scaling of an ample divisor $A$ over $\Spec R$, and
		\begin{enumerate}
			\item if $K_X+\Delta$ is pseudo-effective, then this MMP terminates with a minimal model;
			\item if $K_X+\Delta$ is not pseudo-effective, then this MMP terminates with a Mori fiber space.
		\end{enumerate}
	\end{thm}

	Since we still do not know the existence of Mori fiber spaces for threefolds over imperfect fields, we can only obtain a weaker result in purely positive characteristic.
	
	\begin{thm}\label{thm: main thm char p}
		Let $(\mathcal{X}, \Phi)$ be a four-dimensional $\bQ$-factorial dlt pair projective over $C$, where $C$ is a smooth curve defined over a perfect field with characteristic $p>5$. Let $R := \cO_{C,s}$ where $s \in C$ is a closed point, and $(X, \Delta) := (\mathcal{X} , \Phi) \times_C \Spec R$. We still use $s$ to denote the special point of $\Spec R$. Let $\phi: X \to \Spec R$ be the natural morphism.\par
		Suppose that $\Supp(\phi^{-1}(s)) \subseteq \lfloor \Delta \rfloor$. If $K_X+\Delta$ is pseudo-effective, then we can run a $(K_X+\Delta)$-MMP with scaling of an ample divisor $A$ over $\Spec R$ which terminates with a minimal model.
	\end{thm}
	
	Note that \cite{HW20} proved the termination of MMP when $\kappa(K_X+\Delta\,/\,\Spec R)\ge0$ and $\Delta$ has standard coefficients. They use the effectivity to deduce the termination as in \cite{AHK07}. In fact, in this case it is known that any $(K_X+\Delta)$-MMP terminates.
	
	\medskip
	
	With some extra effort, we can extend Theorem \ref{thm: main thm mixed} to the case where $R$ is a Dedekind domain instead of a discrete valuation ring.
	
	\begin{cor}\label{cor: log terminal model in a neighborhood of s}
		Let $(X, \Delta)$ be a four-dimensional $\bQ$-factorial dlt pair projective over an excellent Dedekind scheme $V$. Assume that $X_{\Qq}\neq\emptyset$ and that every residue field of $V$ is perfect.
		Let $s \in V$ be a closed point and let $\phi: X \to V$ be the natural morphism.
		
		Suppose that $\Supp(\phi^{-1}(s)) \subseteq \lfloor \Delta \rfloor$ and that $k(s)$ has characteristic $p>5$. Then there exists an open neighborhood $U$ of $s$ in $V$ such that we can run a $(K_X+\Delta)$-MMP over $U$ which terminates with either a minimal model or a Mori fiber space.
	\end{cor}
	
	With additional conditions on $X\to V$, we can even run a $(K_X+\Delta)$-MMP globally over $V$ in Proposition \ref{prop: MMP of pair such that (X,X_s) is dlt}. For example, as a special case, we have
	
	\begin{thm}\label{thm: MMP of strictly semi-stable fourfolds}
		Let $V$ be an excellent Dedekind scheme whose residue fields do not have characteristic 2, 3 or 5. Let $X$ be a strictly semi-stable and projective $V$-variety of relative dimension 3. Assume that $X_\Qq\neq\emptyset$ and that every residue field of $V$ is perfect. Then we can run a $K_{X}$-MMP over $V$ which terminates with either a minimal model or a Mori fiber space.
	\end{thm}
	
	Using the same strategy, we can obtain a similar result in positive characteristic from Proposition \ref{prop: MMP for big pair over C in char>0}.
	
	\begin{thm}\label{thm: MMP for strictly semi-stable fourfolds over C in char>0}
		Let $C$ be a smooth curve over a perfect field with characteristic $p>5$. Let $X$ be a strictly semi-stable and projective $C$-variety of relative dimension 3. Assume $K_X$ is big over $C$. Then we can run a $K_X$-MMP over $C$ which terminates with a good minimal model.
	\end{thm}
	
	\noindent\textbf{Acknowledgement}. The authors would like to thank their advisor Christopher D. Hacon for introducing the problem, giving encouragements and answering questions. The authors would also like to thank Jakub Witaszek for reading the paper and answering questions, and Jingjun Han and Jihao Liu for useful discussions. The authors were partially supported by NSF research grants no: DMS-1801851, DMS-1952522 and by a grant from the Simons Foundation; Award Number: 256202. Finally, the authors are grateful to the referees for many valuable comments and suggestions.

	\section{Preliminaries}
	A scheme $X$ will be called a \emph{variety} over a field $k$ (resp. over $\Spec R$, where $R$ is a discrete valuation ring, or DVR for short) if it is integral, separated, and of finite type over $k$ (resp. $\Spec R$).
	
	We refer the reader to \cite{KM98,Kol13} for the standard definitions and results in the Minimal Model Program and to \cite{BMP+20} for those in mixed characteristic.
	
	In this paper, a \emph{pair} $(X,B)$ consists of a normal variety $X$ and an effective $\bR$-divisor $B$ such that $K_X+B$ is $\bR$-Cartier. The pair $(X,B)$ is \emph{Kawamata log terminal (klt)} (resp. \emph{log canonical (lc)}) if for any proper birational morphism $f:X\to Y$ and any prime divisor $E$ on $Y$ we have $\mult_E(B_Y)<1$ (resp. $\mult_E(B_Y )\le1$), where $K_Y+B_Y=f^*(K_X+B)$. If $(X,B)$ admits a log resolution $f:Y\to X$, then it suffices to check the above condition for all prime divisors $E$ on $Y$.
	
	For a pair $(X,B)$ such that $B=\sum b_iB_i$ with $0\le b_i\le1$, we say that $(X,B)$ is \emph{divisorially log terminal (dlt)} if there exists an open subset $U\subseteq X$ such that $U$ is smooth and $\Supp(B|_U)$ is simple normal crossing, and for every proper birational morphism $f:Y\to X$ and any prime divisor $E$ on $Y$ with center $Z$ contained in $X\setminus U$, we have $\mult_E(B_Y)<1$. We say that $a_E(X,B):=1-\mult_E(B_Y)$ is the \emph{log dicrepancy} of $(X, B)$ along $E$. A pair $(X,S+B)$ with $\lf S + B\rf=S$ irreducible, is \emph{purely log terminal (plt)} if $a_E(X, B)> 0$ for any $E\neq S$.\par
	
	A morphism of schemes $f:X\to Y$ is a \emph{universal homeomorphism} if for any morphism $Y'\to Y$, the induced morphism $X'=X\times_YY'\to Y'$ is a homeomorphism. We say that a variety $X$ is \emph{normal up to universal homeomorphism} if its
	normalization $X^\nu\to X$ is a universal homeomorphism.
	
	We say that $\Delta$ has \emph{standard coefficients} if the coefficients of $\Delta$ are contained in $\{1\}\cup\{1-\frac{1}{m}~|~m\in\Zz_{>0}\}$.
	
	We say that a morphism of schemes $f:X\to Z$ is a \emph{contraction} if it is proper, surjective, and $f_*\Oo_X=\Oo_Z$.
	
	\medskip
	
	We first give the definition of strictly semi-stable morphisms.
	
	\begin{defn}(strictly semi-stable)
		\begin{enumerate}
			\item Let $V$ be the spectrum of a DVR $R$. Let $\omega$ be a uniformizer
			of $R$. A flat $V$-variety $X$ of relative dimension $n$ is called \emph{strictly semi-stable} if the following hold.
			\begin{itemize}
				\item The generic fiber $X_\eta$ is smooth, where $\eta\in V$ is the generic point.
				\item For any closed point $x$ in the special fiber $X_s$, there exists a Zariski open neighborhood $U$ of $x$ such that $U$ is \'etale over the scheme
				$\Spec R[X_0,...,X_n]/(X_1\cdots X_m-\omega)$ for some $m\le n$.
			\end{itemize}
			
			As in \cite[2.16]{dJ96}, if $R$ has a perfect residue field, the above definition is equivalent to that $(X,X_s)$ is a simple normal crossing pair.
			\item Let $V$ be a Dedekind scheme. An integral flat quasi-projective $V$-variety $X$ of relative dimension $n$ is called \emph{strictly semi-stable} if $X_{\mathcal{O}_{V,s}}\to \Spec\mathcal{O}_{V,s}$ is strictly semi-stable for any closed point $s\in V$.
		\end{enumerate}
	\end{defn}
	
	\smallskip
	
	\begin{defn}
	    Let $\pi: X \to U$ be a projective morphism of normal varieties. Let $D$ be an $\bR$-divisor on $X$. The \emph{stable base locus} of $D$ over $U$ is the Zariski closed set $\mathbf{B}(D/U)$ given by the intersection of the support of the elements of the real linear system $|D/U|_\bR$. The \emph{augmented base locus} of $D$ over $U$ is the Zariski closed set
	    $$
	    \mathbf{B}_+(D/U) = \mathbf{B}((D-\epsilon A)/U)
	    $$
	    for any ample divisor $A$ over $U$ and any sufficiently small rational number $\epsilon>0$.
	\end{defn}
	
	When running the MMP with scaling of an ample divisor $A$, the ampleness of $A$ will not be preserved. However, the properties that $A \ge 0$ is big and $\mathbf{B}_+(A)$ does not contain any non-klt centers will be preserved by \cite[Lemma 3.10.11]{BCHM10}. These properties are already sufficient in most situations.
	
	\begin{lem}[{\cite[Lemma 2.29]{BMP+20}}, cf. {\cite[Lemma 9.2]{Bir16}}]\label{lem: ample perturb}
	    Let $\pi: (X, \Delta) \to \Spec R$ be a projective morphism from a klt (resp. dlt) pair over a Noetherian local domain. Suppose that $A$ is an ample $\bR$-divisor on $X$. Then there exists an $\bR$-divisor $0 \le A' \sim_\bR A$ such that $(X, \Delta+A')$ is klt (resp. dlt).
	\end{lem}
	
	With the help of Lemma \ref{lem: ample perturb}, we are able to perturb dlt pairs by ample divisors, similar to what we usually do in characteristic 0.
	
	\medskip
	
	
	The following lemma on the existence of dlt modifications is very useful.
	
	\begin{lem}[{\cite[Corollary 4.7]{HW20}}]\label{lem: dlt modification for fourfolds with general coef}
		Let $(X,\Delta)$ be a four-dimensional $\Qq$-factorial log pair defined over a perfect field $k$ of characteristic $p>5$ or a DVR of characteristic $(0,p)$ for $p>5$ with perfect residue field. Then there exists a dlt modification of $(X,\Delta)$, that is, a projective birational morphism $\pi:Y\to X$ such that $(Y,\pi^{-1}_*\Delta+\Exc(\pi))$ is dlt and $\Qq$-factorial, and $K_Y + \pi^{-1}_*\Delta+\Exc(\pi)$ is nef over $X$.
	\end{lem}

	We state a result on the semiampleness in mixed characteristic from \cite{Wit21}.
	
	\begin{thm}[{\cite[Theorem 1.2]{Wit21}}]\label{thm: wit21 1.2}
		Let $X$ be a scheme admitting a proper morphism $\pi: X \to S$ to an excellent scheme $S$ and let $L$ be a line bundle on $X$. Then $L$ is semiample if and only if $L|_{X_\bQ}$ is semiample and $L|_{X_s}$ is semiample for every point $s \in S$ having positive characteristic residue field.
	\end{thm}
	
	
	
	In order to generalize some results from algebraically closed fields to perfect fields, we need the following base change result.
	
	\begin{lem}\label{lem: base change from perfect field to algebraically closed field}
		Let $(X,\Delta)$ be a dlt (resp. klt, plt, lc) geometrically integral log pair over a perfect field $k$. Then $(X_{\bar{k}},\Delta_{\bar{k}})$ is also dlt (resp. klt, plt, lc) over $\bar{k}$, where $X_{\bar{k}}$ and $\Delta_{\bar{k}}$ are the base changes of $X$ and $\Delta$ to $\bar{k}$ respectively.
	\end{lem}
	
	\begin{proof}
		Let $f: Y\to X$ be a log resolution of $(X,\Delta)$. Since $\bar{k}/k$ is a separable extension, $\Spec\bar{k}\to\Spec k$ is a smooth morphism. Therefore $X_{\bar{k}}$ is still normal and $f_{\bar{k}}: Y_{\bar{k}}\to X_{\bar{k}}$ is also a log resolution of $(X_{\bar{k}},\Delta_{\bar{k}})$. Then the statement follows immediately.
	\end{proof}
	
	The following lemma will be used frequently in Section \ref{section: global MMP}, which helps us extend the semiampleness from a point to a neighborhood of it.
	
	\begin{lem}\label{lem: semiample is an open condition}
		Let $f:X\to T$ be a proper morphism of varieties and $\Ff$ a coherent sheaf on $X$. For $t\in T$, let $X_{\mathcal{O}_{T,t}}$ and $\Ff_{\mathcal{O}_{T,t}}$ be the base changes of $X$ and $\Ff$ to $\Spec \mathcal{O}_{T,t}$ respectively. Assume that $\Ff_{\mathcal{O}_{T,t}}$ is generated by global sections. Then $\Ff$ is generated by global sections in a neighborhood of $t\in T$.
	\end{lem}
	\begin{proof}
		Since $f_*\Ff$ is coherent, the canonical map
		$$\phi: f^*f_*\Ff\to \Ff$$
		is a homomorphism of coherent $\Oo_{X}$-modules. By the assumption, the support of $\coker \phi$ is disjoint from the fiber $f^{-1}(t)$, and thus does not intersect $f^{-1}(U)$ for an open neighborhood $U\ni t$ since $f$ is proper. Hence $\phi$ is surjective on $f^{-1}(U)$, which implies the conclusion.
	\end{proof}
	
	\begin{lem}\label{lem: semi-stable in a neighborhood in mixed char}
		Let $(X,\Delta)$ be a dlt pair projective over an excellent Dedekind scheme $V$. Assume that $X_{\Qq}\neq\emptyset$ and that $(X,\Delta)$ admits a log resolution. Then there exists an open subset $U$ of $V$ such that for any closed point $s\in U$, $X_s$ is reduced and $(X,\Delta+X_s)$ is dlt.
	\end{lem}
	\begin{proof}
		Let $f: Y\to X$ be a log resolution of $(X,\Delta)$ and $D$ the reduced divisor whose support is $\Exc(f)\cup f_*^{-1}\Delta$. We may assume that $D_v=0$, where $D_v$ is the vertical part of $D$. Since the strata of $(Y,D)$ are relatively smooth over the generic point of $V$, by the openness of smoothness (see \cite[Tag 01V7]{Sta}), there exists an open subset $U$ of $V$ such that $(Y_U,D_U)$ are relatively smooth over $U$. Then our claim immediately follows.
	\end{proof}

	\section{Special termination}
	
	In this section we prove the special termination of MMP with scaling, assuming the termination of MMP with scaling in lower dimensions. This is actually sufficient to imply termination in our case.
	
	\medskip
	
	First let us recall results about the termination of flips for three-dimensional MMP with scaling in positive and mixed characteristic (cf. \cite{BW17,HNT20,BMP+20}).
	
	\begin{thm}\label{thm: termination of flips with scaling of A for dlt threefolds}
		Let $T$ be a quasi-projective scheme over a perfect field $k$ of characteristic $p>5$. Let $(S,B+A)$ be a three-dimensional $\mathbb Q$-factorial dlt pair projective over $T$. Assume that $K_S+B+A$ is nef over $T$, $A\ge0$ is big over $T$, and that $\mathbf{B}_+(A/T)$ contains no non-klt centers of $(S,B+A)$. Then we can run a $(K_S+B)$-MMP with scaling of $A$ over $T$, and any sequence of steps of such an MMP terminates.
	\end{thm}

	\begin{proof}
		We can run a $(K_S+B)$-MMP with scaling by \cite[Theorem 1.3, Theorem 1.4 and Theorem 4.11]{HNT20} (see also \cite[Theorem 6.2]{HNT20} and Lemma \ref{lem: base change from perfect field to algebraically closed field}). So it remains to show the termination.
		
		When $K_S+B$ is pseudo-effective over $T$, the statement follows directly from \cite[Theorem 6.11]{HNT20} and Lemma \ref{lem: base change from perfect field to algebraically closed field}.
		
		Suppose that $K_S+B$ is not pseudo-effective over $T$. As in the proof of \cite[Corollary 6.10]{HNT20}, we can reduce to the case where $k$ is algebraically closed. Since $\mathbf{B}_+(A/T)$ contains no non-klt centers of $(S,B+A)$, we may write $A\sim_{\mathbb R,T} H+E$ where $H$ is ample and $E$ is effective such that $E$ contains no non-klt centers of $(S,B+A)$. Let $0<\epsilon \ll 1$ such that $K_S+B+\epsilon A$ is not pseudo-effective. Then we can find a boundary $\Delta_{\epsilon}\sim_{\mathbb R,T}B+\epsilon H+\epsilon E$ such that $(S,\Delta_{\epsilon}+(1-\epsilon)A)$ is klt. By \cite[Theorem 1.5]{BW17}, any $(K_S+\Delta_\epsilon)$-MMP with scaling of $(1-\epsilon)A$ over $T$ terminates. For any $(K_S+B)$-MMP with scaling of $A$, since $K_S+B+\epsilon A$ is not pseudo-effective, the scaling number $\lambda$ is always greater than $\epsilon$. Therefore any $(K_S+B+\lambda A)$-trivial and $(K_S+B)$-negative curve is also a $(K_S+\Delta_\epsilon+(\lambda-\epsilon) A)$-trivial and $(K_S+\Delta_\epsilon)$-negative curve, which implies that this MMP is also a $(K_S+\Delta_\epsilon)$-MMP with scaling of $(1-\epsilon)A$. Hence it terminates and the statement follows.
	\end{proof}

	\begin{thm}\label{thm: MMP for dlt threefolds in mixed char}
		Let $R$ be a finite-dimensional excellent ring admitting a dualizing complex, and $T$ a quasi-projective scheme over $R$ whose residue fields do not have characteristic 2, 3 or 5 (cf. \cite[Setting 9.1]{BMP+20}). Let $(S,B+A)$ be a three-dimensional $\mathbb Q$-factorial dlt pair projective over $T$ such that the image of $X$ in $T$ has positive dimension. Assume that $K_S+B+A$ is nef over $T$, $A\ge0$ is big over $T$, and that $\mathbf{B}_+(A/T)$ contains no non-klt centers of $(S,B+A)$. Then we can run a $(K_S+B)$-MMP with scaling of $A$ over $T$, and any sequence of steps of such an MMP terminates.
	\end{thm}
	\begin{proof}
		When $K_S+B$ is pseudo-effective over $T$, the statement follows directly from \cite[Proposition 9.18]{BMP+20}.\par
		Suppose that $K_S+B$ is not pseudo-effective over $T$. As in the proof of Theorem \ref{thm: termination of flips with scaling of A for dlt threefolds}, we can find some $0<\epsilon \ll 1$ such that $K_S+B+\epsilon A$ is not pseudo-effective and a boundary $\Delta_{\epsilon}\sim_{\mathbb R,T}B+\epsilon H+\epsilon E$ such that $(S,\Delta_{\epsilon}+(1-\epsilon)A)$ is klt. Then we can run a $(K_S+\Delta_\epsilon)$-MMP with scaling of $(1-\epsilon)A$ over $T$ by \cite[Theorem 9.27, Theorem 9.32 and Theorem 9.12]{BMP+20}. Any sequence of steps of such an MMP terminates by the proof of \cite[Proposition 4.3]{BW17} and \cite[Theorem 9.33]{BMP+20}. Since any $(K_S+B)$-MMP with scaling of $A$ is also a $(K_S+\Delta_\epsilon)$-MMP with scaling of $(1-\epsilon)A$, the statement follows.
	\end{proof}

	\begin{rem}
		In view of \cite{HW19} and \cite{XX22}, Theorem \ref{thm: termination of flips with scaling of A for dlt threefolds} and \ref{thm: MMP for dlt threefolds in mixed char} are also expected to hold for $p=5$. However, many of the references that we use in our proofs require $p>5$.
	\end{rem}
	
	\medskip
	
	Then we can prove the special termination of MMP with scaling for fourfolds in positive and mixed characteristic. We closely follow the approach in \cite{Fuj07}, and the major difference is that we consider an MMP with scaling rather than a general MMP.
	
	\begin{thm}\label{thm: special termination char p}
		Let $X$ be a normal $\bQ$-factorial fourfold over a perfect field $k$ with char $p>5$, and $B$ an effective $\bR$-divisor such that $(X, B)$ is dlt. Let
		$$
		X_0 \dashrightarrow X_1 \dashrightarrow X_2 \dashrightarrow \cdots \dashrightarrow X_i \dashrightarrow \cdots,
		$$
		be a sequence of $(K_X+B)$-flips with scaling of an ample divisor $A$. Then after finitely many flips the flipping locus and the flipped locus are disjoint from $\lfloor B_i \rfloor$, where $B_i$ is the strict transform of $B$ on $X_i$.
	\end{thm}
	
	In the following we will always use $\bar{S}$ to denote a non-klt center of the pair $(X, B)$, i.e. a stratum of $\lfloor B \rfloor$, and use $\nu: S \to \bar{S}$ to denote the normalization. Note that $\nu$ is homeomorphic since $\bar{S}$ is normal up to universal homeomorphism (\cite[Lemma 2.28]{BMP+20}).
	Now we can apply adjunction for dlt pairs
	to $S$ and get a boundary $B_S$ on $S$ such that the pair $(S, B_S)$ is also dlt, where $(K_X+B)|_S = K_S+B_S$.
	
	\begin{defn}
		Let $I \subseteq [0,1]$ be a coefficient set. Then we define
		$$
		D(I) := \{1 - \frac{1}{m} + \sum_j \frac{r_j b_j}{m} \mid m\in \bZ_{>0}, r_j \in \bZ_{\ge 0}, b_j \in I\} \cap [0,1].
		$$
	\end{defn}
	
	\begin{lem}[{\cite[7.4.4 Lemma]{K+92}}]\label{lem: finiteness of coefficients}
		For any fixed numbers $0 < b_j \le 1$ and $c>0$, there are only finitely many possible values for $m \in \bZ_{>0}$ and $r_j \in \bZ_{\ge 0}$ such that
		$$
		1 - \frac{1}{m} + \sum_j \frac{r_j b_j}{m} \le 1-c.
		$$
	\end{lem}
	
	\begin{defn}
		Let $\bar{S}$ be a non-klt center of the dlt pair $(X, B)$ and $S\to\bar{S}$ the normalization. We define the difficulty
		$$
		d_I(S, B_S) := \sum_{\alpha \in D(I)} \# \{E \mid a(E, B_S) < -\alpha\ \text{and}\ c_S(E) \not\subseteq \lfloor B_S \rfloor\}.
		$$
	\end{defn}
	
	We know that $d_I(S, B_S) < \infty$ by Lemma \ref{lem: finiteness of coefficients}.
	
	
	\begin{defn}
		Let $f: X \to Y$ be a birational contraction with $\dim X \ge 2$. We say that $f$ is of type (D) if $f$ contracts at least one divisor, and that $f$ is of type (S) if $f$ is an isomorphism in codimension one.
		
		Let $X\stackrel{f}{\longrightarrow} Y \stackrel{g}{\longleftarrow} Z$ be a pair of birational contractions. We say this is of type (DS) if $f$ is of type (D) and $g$ is of type (S). Type (DD), (SD), (SS) are defined similarly.
	\end{defn}
	
	\medskip
	
	\begin{proof}[Proof of Theorem \ref{thm: special termination char p}]
		First we fix some notation. Suppose that $\psi_i: X_i \to Z_i$ is the small contraction, and that $X_{i+1} \to Z_i$ is the corresponding flip. For some fixed non-klt center $\bar{S}_i$, let $S_i$ be the normalization of $\bar{S}_i$ and $T_i$ the normalization of $\psi_i(\bar{S}_i)$
		
		\bigskip
		
		We divide the proof into 3 steps.
		
		\bigskip
		
		\textbf{Step 1.} We claim that after finitely many flips, the flipping locus does not contain any non-klt centers. First note that the number of non-klt centers is finite. If there is a non-klt center contained in the flipping locus, then by the negativity lemma (\cite[Lemma 3.38]{KM98}) the number of non-klt centers will decrease, which can only happen for finitely many times.
		
		\medskip
		
		Therefore, we may assume that the flipping locus contains no non-klt centers of the pair $(X_i, B_i)$ for every $i$. Then $\varphi_i: X_i \dashrightarrow X_{i+1}$ induces a birational map $\varphi_i|_{S_i}: S_i \dashrightarrow S_{i+1}$, where $\bar{S}_i$ is a non-klt centers of $(X_i, B_i)$ and $\bar{S}_{i+1}$ is the corresponding non-klt center of $(X_{i+1}, B_{i+1})$. For simplicity we will omit $|_{S_i}$ if there is no danger of confusion.
		
		\smallskip
		
		The next lemma will be used repeatedly in the next step.
		
		\begin{lem}\label{lem: a increase}
			Under the above assumptions, we have
			$$
			a(E, S_i, B_{S_i}) \le a(E, S_{i+1}, B_{S_{i+1}})
			$$
			for every valuation $E$. In particular,
			$$
			\mathrm{totaldiscrep}(S_i, B_{S_i}) \le \mathrm{totaldiscrep}(S_{i+1}, B_{S_{i+1}}).
			$$
		\end{lem}
		
		\begin{proof}
			Let $p: W \to X_i$ and $q: W \to X_{i+1}$ be a common resolution of $\varphi_i: X_i \dashrightarrow X_{i+1}$, and let $R$ be the strict transform of $S_i$. By the negativity lemma,
			$$
			p^*(K_{X_i}+B_i) \ge q^*(K_{X_{i+1}}+B_{i+1}).
			$$
			Restricting to $R$, we have
			$$
			p^*(K_{S_i}+B_{S_i}) \ge q^*(K_{S_{i+1}}+B_{S_{i+1}}),
			$$
			which implies the lemma.
		\end{proof}
		
		\medskip
		
		\textbf{Step 2.} In this step, we show that after finitely many flips, $\varphi_i$ induces an isomorphism of log pairs $(S_i, B_{S_i}) \simeq (S_{i+1}, B_{S_{i+1}})$ for every non-klt center $\bar{S_i}$. Here by an isomorphism of log pairs we mean that $\varphi_i: S_i \to S_{i+1}$ is an isomorphism with $\varphi_{i *}(B_{S_i})=B_{S_{i+1}}$.
		
		
		We prove by induction on the dimension $d$ of the non-klt center $\bar{S_i}$. When $d=0$, it is obvious. When $d=1$, the claimed isomorphism of log pairs follows from Lemma \ref{lem: a increase} and the fact that $d_I(S_i, B_{S_i}) < \infty$, since if $\varphi_i$ is not an isomorphism then $d_I(S_i, B_{S_i})$ will decrease, which can only happen for finitely many times.
		
		\medskip
		
		First we reduce to the case where, after finitely many flips, all $S_i \to T_i \leftarrow S_{i+1}$ are of type (SS). We achieve this by the following lemma.
		
		\begin{lem}[{\cite[Proposition 4.2.14]{Fuj07}}]
			Under the above assumptions, we have the inequality $d_I(S_i, B_{S_i}) \ge d_I(S_{i+1}, B_{S_{i+1}})$.
			
			Moreover, if $S_i \to T_i \leftarrow S_{i+1}$ is of type (SD) or (DD), then $d_I(S_i, B_{S_{i}}) > d_I(S_{i+1}, B_{S_{i+1}})$.
			
			Therefore, after finitely many flips, all $S_i \to T_i \leftarrow S_{i+1}$ are of type (SS) or (DS).
		\end{lem}
		
		Hence by shifting the index $i$, we may assume that $S_i \to T_i \leftarrow S_{i+1}$ is of type (SS) or (DS) for all $i \ge 0$. Since the number of type (DS) is bounded by the Picard number of $S_0$, there are only finitely many $i \ge 0$ such that $S_i \to T_i \leftarrow S_{i+1}$ is of type (DS). Therefore by further shifting the index $i$ we may assume that $S_i \to T_i \leftarrow S_{i+1}$ is of type (SS) for all $i \ge 0$.
		
		\medskip
		
		The next lemma guarantees that the coefficients of $B_{S_i}$ eventually become stationary.
		
		\begin{lem}[{\cite[Lemma 4.2.15]{Fuj07}}]
			By shifting the index $i$, we may assume that $a(E, S_i, B_{S_i}) = a(E, S_{i+1}, B_{S_{i+1}})$ for every $i$ if $E$ is a divisor on both $S_i$ and $S_{i+1}$.
		\end{lem}
		
		To summarize, by shifting the index $i$, we may assume that $S_i \to T_i \leftarrow S_{i+1}$ is of type (SS) and that $\varphi_{i *}(B_{S_i})=B_{S_{i+1}}$ for all $i \ge 0$.
		
		\medskip
		
		Denote by $\lambda_i$ the scaling number in each step, i.e. $\lambda_i := \inf\{\lambda \in \bR \mid K_{X_i} + B_i + \lambda A_i\ \mathrm{is\ nef}\}$. Since $A=A_0$ is ample, we may choose a general member of $|A|_{\bQ}$ and assume that $(X_0, B_0 + \lambda_0 A_0)$ is dlt. By adjunction for dlt pairs, $(S_0, B_{S_0} + \lambda_0 A_{S_0})$ is also dlt.
		
		
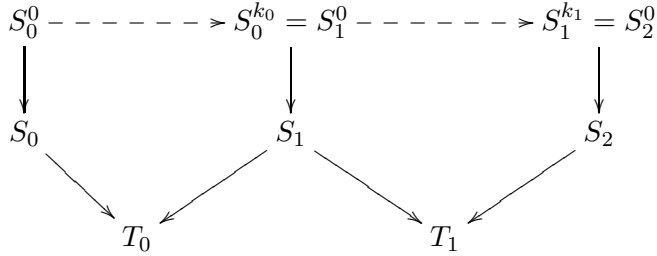
\[
		\xymatrix{
			S_0^0 \ar@{-->}[rr] \ar[d] & & S_0^{k_0} = S_1^0 \ar[d] \ar@{-->}[rr] & & S_1^{k_1} = S_2^0 \ar[d] \\
			S_0 \ar[dr] & & S_1 \ar[dl] \ar[dr] & & S_2 \ar[dl] \\
			& T_0 & & T_1 &
		}
		\]
		
		Let $f: S_0^0 \to S_0$ be a small $\bQ$-factorial dlt modification. Then run a $(K_{S_0^0} + B_{S_0^0})$-MMP over $T_0$ with scaling of $\lambda_0 A_{S_0^0}$, where $K_{S_0^0} + B_{S_0^0} = f^*(K_{S_0} + B_{S_0})$ and $A_{S_0^0} = f^* A_{S_0}$. By Theorem \ref{thm: termination of flips with scaling of A for dlt threefolds} this MMP terminates with a minimal model $(S_0^{k_0}, B_{S_0^{k_0}})$. Since $K_{S_1} + B_{S_1}$ is ample over $T_0$, $(S_1, B_{S_1})$ is an ample model of $(S_0^0, B_{S_0^0})$ over $T_0$. Therefore $S_0^{k_0} \to T_0$ factors through $S_1$. Since $K_{S_1} + B_{S_1} + \lambda_1 A_{S_1}$ is nef, so is $K_{S_0^{k_0}} + B_{S_0^{k_0}} + \lambda_1 A_{S_0^{k_0}}$. Set $S_1^0 := S_0^{k_0}$ and continue.
		
		Then we get a sequence of steps of $(K_{S_0^0} + B_{S_0^0})$-MMP:
		$$
		S_0^0 \dashrightarrow S_0^1 \dashrightarrow \cdots \dashrightarrow S_0^{k_0} = S_1^0 \dashrightarrow S_1^1 \dashrightarrow \cdots \dashrightarrow S_i^0 \dashrightarrow \cdots.$$
		We claim that this MMP is actually a $(K_{S_0^0} + B_{S_0^0})$-MMP with scaling of $\lambda_0 A_{S_0^0}$. Note that $K_{S_0^0} + B_{S_0^0} + \lambda_0 A_{S_0^0}$ is trivial over $T_0$, as $K_{X_0} + B_0 + \lambda_0 A_0$ is trivial over $Z_0$. Thus the MMP $S_0^0 \dashrightarrow S_1^0$ is $(K_{S_0^0} + B_{S_0^0} + \lambda_0 A_{S_0^0})$-trivial, and so $\lambda_0$ is always the relative nef threshold. Therefore the above sequence is a $(K_{S_0^0} + B_{S_0^0})$-MMP with scaling of $\lambda_0 A_{S_0^0}$ (globally).
		
		By the termination of MMP with scaling for threefolds in positive characteristic (Theorem \ref{thm: termination of flips with scaling of A for dlt threefolds}), $S_i^0 = S_j^0$ for $i,j$ sufficiently large.
		
		\smallskip
		
		Finally we show that $S_i = S_{i+1}$ for $i \gg 0$. This is because $K_{S_i} + B_{S_i}$ and $K_{S_{i+1}} + B_{S_{i+1}}$ have the same pullback on $S_i^0 = S_{i+1}^0$ and are relatively anti-ample and ample respectively, which is impossible if $S_i \neq S_{i+1}$.
		
		\bigskip
		
		\textbf{Step 3.} By Step 2, after finitely many flips, $(S_i, B_{S_i}) \simeq (S_{i+1}, B_{S_{i+1}})$. Applying the negativity lemma to $S_i \to T_i \leftarrow S_{i+1}$ we see that $\bar{S_i}$ contains no flipping curves and $\bar{S}_{i+1}$ contains no flipped curves. In particular, $\lfloor B_i \rfloor$ contains no flipping curves and no flipped curves. As a result, $\lfloor B_i \rfloor$ cannot contain the whole flipping locus. If the flipping locus intersects $\lfloor B_i \rfloor$, then there exists a flipping curve $C$ such that $C \cdot \lfloor B_i \rfloor > 0$. Hence $\lfloor B_{i+1} \rfloor$ intersects every flipped curve negatively. So $\lfloor B_{i+1} \rfloor$ contains a flipped curve, which is a contradiction. Therefore the flipping locus is disjoint from $\lfloor B_i \rfloor$. Similarly for the flipped locus.
	\end{proof}
	
	\smallskip
	
	\begin{rem}
		The proof of Theorem \ref{thm: special termination char p} actually applies to higher dimensions. Roughly speaking, if we know the termination of MMP with scaling in dimension $\le n-1$ (under some conditions), then we can deduce the special termination of MMP with scaling in dimension $n$ (under the same conditions).
	\end{rem}
	
	We can use the same method to deduce the special termination of MMP with scaling in mixed characteristic.
	
	\begin{thm}\label{thm: special termination mixed}
		Let $X$ be a normal $\bQ$-factorial fourfold over a DVR $R$ of mixed characteristic whose residue field $k$ is perfect with char $p>5$, and $B$ an effective $\bR$-divisor such that $(X, B)$ is dlt. Let
		$$
		X_0 \dashrightarrow X_1 \dashrightarrow X_2 \dashrightarrow \cdots \dashrightarrow X_i \dashrightarrow \cdots,
		$$
		be a sequence of $(K_X+B)$-flips with scaling of an ample divisor $A$. Then after finitely many flips the flipping locus and the flipped locus are disjoint from $\lfloor B_i \rfloor$, where $B_i$ is the strict transform of $B$ on $X_i$.
	\end{thm}
	
	\begin{proof}
		The proof is exactly the same as the proof of Theorem \ref{thm: special termination char p}, except that in Step 2 we use Theorem \ref{thm: MMP for dlt threefolds in mixed char} instead of Theorem \ref{thm: termination of flips with scaling of A for dlt threefolds} to conclude the termination of the $(K_{S^0_0}+B_{S^0_0})$-MMP if $S_0$ is not supported on the special fiber.
	\end{proof}

	\section{Relative MMP over DVRs}
	
	In this section we prove Theorem \ref{thm: main thm mixed} and \ref{thm: main thm char p}. First we prove the following base point free theorem (see \cite{Pos21, BBS21} for other recent results on abundance in mixed characteristic). This is similar to \cite[Proposition 5.1]{HW20} and we actually use the similar strategy in the proof. However, instead of assuming $K_X+\Delta$ is nef and big we assume that the boundary $\Delta$ is big. Actually this is enough for the purpose of running MMP with scaling of an ample divisor. Besides, under this new assumption we are able to simplify the proof by using the results in \cite{HNT20}.
	
	\begin{prop}\label{prop: Kx semiample mixed char}
		Let $(X, \Delta)$ be a four-dimensional $\bQ$-factorial dlt pair and $g:X\to Z$ a projective contraction, where $Z$ a projective variety over a DVR $R$. Suppose that $R$ has perfect residue field of characteristic $p > 5$. Let $s, \eta \in \Spec R$ be the special and the generic point respectively, and let $\phi: X \to \Spec R$ be the natural morphism. Suppose that $\lfloor \Delta \rfloor = \Supp(\phi^{-1}(s))$ and one of the following conditions holds:
		\begin{enumerate}
			\item $R$ is of mixed characteristic $(0,p)$, $K_X + \Delta$ is $g$-nef, and $\Delta$ is $g$-big;
			\item $R$ is purely of positive characteristic $p$, $K_X + \Delta$ is $g$-nef and $g$-big, and $\Delta$ is $g$-big.
		\end{enumerate}
		Then $K_X + \Delta$ is $g$-semiample.
	\end{prop}
	
	\begin{proof}
		We only prove the case $Z=\Spec R$ as the general case is quite similar.
		
		Write $\Supp X_s = \sum_{i=1}^r E_i$ for irreducible divisors $E_i$. Notice that $(X_\eta,\Delta_{\eta})$ is klt and $\Delta_\eta$ is big. By the abundance theorem for threefolds in characteristic 0 or \cite[Theorem 1.4]{DW19}, $(K_X+\Delta)|_{X_\eta}$ is semiample. Then by Theorem \ref{thm: wit21 1.2}, it suffices to show that $(K_X + \Delta)|_{X_s}$ is semiample.
		
		Since $\Delta$ is big, we may write $\Delta \sim_\bQ H + F + G$ where $H$ is ample, $F$ and $G$ are effective, and $F$ is supported on $X_s$ while the support of $G$ contains no divisors in $X_s$. Let $\pi: Y \to X$ be a dlt modification of $(X, \Delta + \delta G)$ for some $0 < \delta \ll 1$ (see Lemma \ref{lem: dlt modification for fourfolds with general coef}). Let $\Delta_Y = \pi_*^{-1} \Delta + \Ex(\pi)$. Then $K_Y + \Delta_Y = \pi^*(K_X + \Delta)$ and $\pi_*^{-1} G$ contains no strata of $\lfloor \Delta_Y \rfloor$. Let $P$ be an effective $\pi$-exceptional divisor such that $-P$ is $\pi$-ample. We may assume that $\pi^*H - P$ is ample over $\Spec R$. Note that the support of $P$ is contained in $Y_s$ and that $\Delta_Y - \pi^* \Delta$ is supported on $\Ex(\pi)$. We have
		\begin{align*}
		\Delta_Y &\sim_\bQ \Delta_Y - \pi^*\Delta + \pi^*(H+F+G) - P + P  \\
		&=(\pi^*H - P) + (P + \pi^*F + \Delta_Y - \pi^*\Delta + \pi^*G - \pi_*^{-1}G) + \pi_*^{-1}G.
		\end{align*}
		
		Let $H_Y = \pi^*H-P-aY_s$, $F_Y = P + \pi^*F + \Delta_Y - \pi^*\Delta + \pi^*G - \pi_*^{-1}G + aY_s$, and $G_Y = \pi_*^{-1}G$. Then $H_Y$ is ample, $F_Y$ is supported on $Y_s$, and the support of $G_Y$ contains no divisors in $Y_s$. Furthermore, if we choose $a \gg 0$ then $F_Y$ is effective. Therefore, replacing $X, \Delta, H, F, G$ by $Y, \Delta_Y, H_Y, F_Y, G_Y$, we may assume in addition that $(X, \Delta + \delta G)$ is dlt.
		
		Let $\Delta_{\epsilon, \delta} := (1-\delta)\Delta + \delta(H+F+G) + \epsilon X_s \sim_\bQ \Delta$, where $\delta$ is sufficiently small. Then the pair $(X, \Delta_{\epsilon, \delta})$ is klt for some small but possibly negative $\epsilon$.
		
		Since $H$ is ample, we may further assume that the support of $F = \sum f_i E_i$ equals $X_s$ where $f_i$ are chosen generically. Then fixing $\delta$ and increasing $\epsilon$, we obtain a sequence of rational numbers $\epsilon < \epsilon_1 < \epsilon_2 < \cdots < \epsilon_r$ such that $U_i := \lfloor \Delta_{\epsilon_i, \delta} \rfloor = \sum^i_{j=1} E_j$ and $E_i$ occurs with coefficient one in $\Delta_{\epsilon_i, \delta}$. Here of course we have re-indexed the $E_i$ accordingly.
		\begin{claim}
			$(K_X + \Delta)|_{E^\nu_i}$ is semiample, where $E^\nu_i \to E_i$ is the normalization. Hence $(K_X + \Delta)|_{E_i}$ is also semiample.
		\end{claim}
		\begin{proof}
			Set $K_{E_i^\nu}+\Delta_{E_i^\nu}=(K_X+\Delta)|_{E_i^\nu}$. Since $\Delta$ is big, we may write $\Delta\sim_{\Qq}H'+F'+G'$, where $H'$ is ample, $F'$ is supported on $X_s$ and $G'\ge 0$ contains no divisors in $X_s$. By shifting $F'$ by a multiple of $\phi^{-1}(s)$, we may assume that $F'=E_i+L'$ where $\Supp L'$ does not contain $E_i$. For $0<\epsilon\ll1$, we have 
			\begin{align*}
			K_{X}+\Delta &\sim_{\Qq}(1-\epsilon)(K_X+E_i+\Delta-E_i)+\epsilon (K_X+E_i+L'+H'+G')\\
			&=K_X+E_i+\epsilon(H'+G')+\epsilon L'+(1-\epsilon)(\Delta-E_i).
			\end{align*}
			By the assumption $M'_{\epsilon}:=\epsilon L'+(1-\epsilon)(\Delta-E_i)$ is an effective divisor whose support does not contain $E_i$. Thus we have 
			\begin{align*}
			\Delta_{E_i^\nu}&=(K_X+\Delta)|_{E_i^\nu}-K_{E_i^\nu}\\
			&\sim_{\Qq}((K_X+E_i)|_{E_i^\nu}-K_{E_i^\nu})+\epsilon(H'+G')|_{E_i^\nu}+M'_{\epsilon}|_{E_i^\nu}.
			\end{align*}
			Since both $(K_X+E_i)|_{E_i^\nu}-K_{E_i^\nu}$ and $(\epsilon G'+M'_\epsilon)|_{E_i^\nu}$ are effective, $\Delta_{E_i^\nu}$ is big. Hence by \cite[Theorem 1.4(3)]{HNT20} the nef divisor $K_{E_i^\nu}+\Delta_{E_i^\nu}$ is semiample. Finally, $(K_X+\Delta)|_{E_i}$ is also semiample by \cite[Lemma 2.11(3)]{CT20}, since $E^\nu_i\to E_i$ is a universal homeomorphism.
		\end{proof}
		
		Now by induction we may assume that $(K_X+\Delta)|_{U_{i-1}}$ is semiample and we must show that $(K_X+\Delta)|_{U_i}$
		is semiample. By \cite[Corollary 2.9]{Kee99}, it suffices to show that $g_2|_{U_{i-1}\cap E_i}$ has connected geometric fibers
		where $g_2:E_i^\nu\to V$ is the morphism associated to the semiample $\Qq$-divisor $(K_X+\Delta)|_{E_i^\nu}$. Since $(K_X+\Delta)|_{E_i^\nu}\equiv_V 0$, we have $-(K_{E_i^\nu}+\Delta'_{E_i^\nu}):=-(K_X+\Delta_{\epsilon_i,\delta}-\delta H)|_{E_i^\nu}$
		is ample over $V$. By \cite[Theorem 1.2]{NT20}, the fibers of the non-klt locus of $(E_i^\nu,\Delta'_{E_i^\nu})$ are geometrically connected. This non-klt locus actually coincides with $U_{i-1}\cap E_i$ as $(X, \Delta + \delta G)$ is dlt. Hence the statement of the proposition follows.
	\end{proof}
	
	\begin{prop}\label{prop: cone theorem for semi-stable fourfold over Spec R}
		Let $(X, \Delta)$ be a four-dimensional $\bQ$-factorial dlt pair projective over a DVR $R$, where $R$ has perfect residue field of characteristic $p > 5$. Let $s \in \Spec R$ be the special point and let $\phi: X \to \Spec R$ be the natural morphism. Suppose that $\Supp(\phi^{-1}(s)) \subseteq \lfloor \Delta \rfloor$. When $R$ is purely of positive characteristic, we also assume that it is a local ring of a curve $C$ defined over a perfect field, and that $(X, \Delta) := (\mathcal{X} , \Phi) \times_C \Spec R$ for a four-dimensional $\bQ$-factorial dlt pair $(\mathcal{X} , \Phi)$ which is projective over $C$.\par
		Let $E_1, \cdots, E_r$ be the irreducible components of $\phi^{-1}(s)$ and let $E_i^\nu \to E_i$ be the normalization. Then we have
		\begin{enumerate}
			\item $\sum_{i=1}^r\overline{\mathrm{NE}}(E^\nu_i)\to\overline{\mathrm{NE}}(X/\Spec R)$ is surjective;
			\item $\overline{\mathrm{NE}}(X/\Spec R)=\overline{\mathrm{NE}}(X/\Spec R)_{K_X+\Delta\ge0}+\sum_{1\le i\le r,~j\ge1}\Rr_{\ge0}[\Gamma_{i,j}]$ for countably many curves $\Gamma_{i,j}\subseteq E_i$ such that $(K_X+\Delta)\cdot\Gamma_{i,j}<0$;
			\item $\overline{\mathrm{NE}}(X/\Spec R)=\overline{\mathrm{NE}}(X/\Spec R)_{K_X+\Delta+A\ge0}+\sum_{1\le i\le r,~1\le j\le m_i}\Rr_{\ge0}[\Gamma_{i,j}]$ for finitely many curves $\Gamma_{i,j}\subseteq E_i$ such that $(K_X+\Delta+A)\cdot\Gamma_{i,j}<0$, where $A$ is an ample $\Rr$-divisor.
		\end{enumerate}
	\end{prop}
	\begin{proof}
		Since any curve over the generic point $\eta$ extends to a curve on the special fiber $\phi^{-1}(s)$, (1) follows immediately.\par
		For (2) and (3), we may suppose that $K_X+\Delta$ is not nef over $\Spec R$. Thus it is not nef over $s$. Since $K_X+\Delta$ is dlt and $E_i\subseteq\lf\Delta\rf$, $(E_i^\nu,\Delta_i)$ is also dlt, where $K_{E^\nu_i}+\Delta_i=(K_X+\Delta)|_{E_i^\nu}$. Then the statements follow by the cone theorem for dlt threefolds (\cite[Theorem 1.3]{HNT20}).
	\end{proof}
	
	\begin{prop}\label{prop: L semiample mixed char}
		Let $(X, \Delta)$ be a four-dimensional $\bQ$-factorial dlt pair projective over a DVR $R$, where $R$ 
		has perfect residue field of characteristic $p > 5$. Let $s, \eta \in \Spec R$ be the special and the generic point respectively, and let $\phi: X \to \Spec R$ be the natural morphism.  Let $A$ be a big $\bQ$-divisor on $X$ such that $(X, 
		\Delta+A)$ is dlt and $\mathbf{B}_+(A)$ does not contain any non-klt centers of $(X, \Delta+A)$.\par
		Suppose that $\Supp(\phi^{-1}(s)) \subseteq \lfloor \Delta \rfloor$ and one of the following conditions holds:
		\begin{enumerate}
			\item $R$ is of mixed characteristic $(0,p)$ and $L := K_X + \Delta + A$ is nef;
			\item $R$ is purely of positive characteristic $p$ and $L := K_X + \Delta + A$ is nef and big. 
		\end{enumerate}
		Then $L$ is semiample and induces a morphism $f: X \to W$ over $\Spec R$. In particular, every $f$-numerically trivial $\Qq$-Cartier divisor descends to a $\Qq$-Cartier divisor on $W$.
	\end{prop}
	
	\begin{proof}
		The first statement follows from Proposition \ref{prop: Kx semiample mixed char} by perturbing the boundary $\Delta$ via $A$.
		More explicitly, let $\lfloor \Delta \rfloor =E+F$ where $E$ is supported on $X_s$ while the support of $F$ contains no divisors in $X_s$.
		Since $\mathbf{B}_+(A)$ does not contain any non-klt centers of $(X,\Delta+A)$, we can write $A \sim_\bQ H+G$ where $H$ is ample and $G\ge0$ contains no non-klt centers of $(X,\Delta)$. Hence $(X, \Delta+A+\epsilon G)$ is still dlt for some sufficiently small $\epsilon>0$. Then there exists $\delta >0$ sufficiently small such that $\epsilon H + \delta F$ is ample. Therefore we can choose a general member  $H'\sim_{\Qq}\epsilon H+\delta F$, such that if
		$$\Delta' := \Delta - \delta F + (1-\epsilon) A + \epsilon G + H'\sim_{\Qq}\Delta+A,
		$$
		then $(X, \Delta')$ is dlt and $\lfloor \Delta' \rfloor = \Supp(\phi^{-1}(s))$. Notice that $K_X + \Delta' \sim_\bQ L$.
		Applying Proposition \ref{prop: Kx semiample mixed char} to the pair $(X, \Delta')$ we conclude that $L$ is semiample.
		
		\smallskip
		
		For the last statement, if $M$ is an $f$-numerically trivial $\Qq$-Cartier divisor on $X$, then we claim that $\mathbf{B}_+(A+M+f^*D)\subseteq\mathbf{B}_+(A)$ for some sufficiently ample divisor $D$ on $W$. Indeed, by definition $\mathbf{B}_+(A)=\mathbf{B}(A-\epsilon H)$ for any ample divisor $H$ and any $0<\epsilon\ll1$. Since $\epsilon H+M$ is $f$-ample, $\bar{H}:=\epsilon H+M+f^*D$ is ample for sufficiently ample $D$ on $W$. Then the claim follows as $\mathbf{B}_+(A+M+f^*D)\subseteq\mathbf{B}(A+M+f^*D-\bar{H})=\mathbf{B}(A-\epsilon H)$.\par
		Now we can find $A'\sim_{\Qq}(1-\frac{1}{n})A+\frac{1}{n}(A+M+f^*D)$ such that $(X,\Delta+A')$ is dlt.
		Then by Propostion \ref{prop: cone theorem for semi-stable fourfold over Spec R}(3), $K_X+\Delta+A'+f^*D'$ is nef for sufficiently ample $D'$ on $W$.
		Applying Proposition \ref{prop: Kx semiample mixed char} to $K_X+\Delta+A'+f^*D'$, we see that $K_X+\Delta+A'+f^*D'$ is semiample and hence defines a contraction $f': X \to W'$.
		Since $D'$ is sufficiently ample, $f$ factors through $f'$.
		Since $A'+f^*D'-A$ is $f$-numerically trivial, we see that $f'$ actually coincides with $f$. Thus
		$$
		K_X+\Delta+A'\sim_{\Qq,W}0\sim_{\Qq,W} K_X+\Delta+A.
		$$
		Therefore $M\sim_{\Qq,W}0$ and the statement follows.
	\end{proof}
	
	\begin{rem}
		From the proof we can see that the above proposition also holds in the relative setting, i.e. for a projective contraction $g:X\to Z$ over $\Spec R$. Indeed, if $K_X+\Delta+A$ is nef over $Z$, then by Proposition \ref{prop: cone theorem for semi-stable fourfold over Spec R}(3) $K_X+\Delta+A+g^*D$ is also nef for sufficiently ample $D$ on $Z$. Thus we reduce to the global case.
	\end{rem}
	
	Using the same strategy as in \cite[Theorem 6.3]{Bir16}, we can show the existence of flips by reducing to the case where the coefficients belong to the standard set, which is known by \cite[Proof of Theorem 1.2]{HW20}.
	
	\begin{thm}\label{thm: existence of flips for semi-stable fourfold over curve with general coef}
		Let $(X, \Delta)$ be a four-dimensional $\bQ$-factorial dlt pair projective over a DVR $R$, where $R$ has perfect residue field of characteristic $p > 5$. Let $s \in \Spec R$ be the special point and let $\phi: X \to \Spec R$ be the natural morphism. Suppose that $\Supp(\phi^{-1}(s)) \subseteq \lfloor \Delta \rfloor$. When $R$ is purely of positive characteristic, we also assume that it is a local ring of a curve $C$ defined over a perfect field, and that $(X, \Delta) := (\mathcal{X} , \Phi) \times_C \Spec R$ for a four-dimensional $\bQ$-factorial dlt pair $(\mathcal{X} , \Phi)$ which is projective over $C$.\par
		If $f: X\to Z$ is a $(K_X+\Delta)$-flipping contraction such that $\rho(X/Z)=1$, then the flip of $f$ exists.
	\end{thm}
	
	\begin{proof}
		Let $\zeta(\Delta)$ be the number of components of $\Delta$ whose coefficient is not contained in the standard set $\Gamma:=\{1\}\cup\{1-\frac{1}{n}~|~n>0\}$. We prove the result by induction on $\zeta(\Delta)$.\par
		If $\zeta(\Delta)=0$, then this follows by the proof of \cite[Theorem 1.2]{HW20}.
		Therefore we may assume that $\zeta(\Delta)>0$ and write $\Delta=aS+B$ where $a\notin\Gamma$. Note that $S$ is not contained in $\phi^{-1}(s)$. Let $\nu: Y\to X$ be a log resolution of $(X,S+B)$ such that $\nu$ is an isomorphism at the generic points of strata of $\lf\Delta\rf$ and let $B_Y=\nu^{-1}_*B+\Exc(\nu)$, $S_Y=\nu^{-1}_*S$.
		Since $\zeta(S_Y+B_Y)<\zeta(\Delta)$, we may run a $(K_Y+S_Y+B_Y)$-MMP over $Z$.
		This is because the cone theorem is established in Proposition \ref{prop: cone theorem for semi-stable fourfold over Spec R} and the contraction theorem is also established in Proposition \ref{prop: L semiample mixed char}. This MMP terminates by \cite[Theorem 2.14]{HW20} and we get a minimal model $(W,S_W+B_W)$, where $S_W,B_W$ are the birational transforms of $S_Y,B_Y$ respectively.
		Then we run a $(K_Y+aS_Y+B_Y)$-MMP with scaling of $(1-a)S_Y$ over $Z$.
		Note that this is also a $(K_Y+B_Y)$-MMP and $\zeta(B_Y)<\zeta(\Delta)$.
		Hence we can run such an MMP and it terminates by \cite[Theorem 2.14]{HW20}.
		Thus we obtain a minimal model $(X^+,aS^++B^+)$. Then it is easy to see that $X^+\to Z$ is the desired flip (cf. \cite[Proof of Theorem 1.1]{HW19}).
	\end{proof}

	\begin{proof}[Proof of Theorem \ref{thm: main thm mixed}]
		Suppose that we already have a sequence of steps of $(K_X+\Delta)$-MMP
		$$X = X_0 \dashrightarrow X_1 \dashrightarrow \cdots \dashrightarrow X_k,$$
		and we want to continue this process.
		Denote the special fiber of $X_k$ by $X_{k,s}$, and the birational transforms of $\Delta$ and $A$ on $X_k$ by $\Delta_k$ and $A_k$ respectively.
		Replace $A$ by a general member in $|A|_\bQ$, we may assume that $(X, \Delta + A)$ is dlt and that $\mathbf{B}_+(A)$ does not contain any non-klt centers of $(X,\Delta+A)$. Let $\lambda_i$ be the scaling number in each step, i.e. $\lambda_i := \inf\{\lambda \in \bR \mid K_{X_i} + \Delta_i + \lambda A_i\ \mathrm{is\ nef}\}$. If $K_{X_k} + \Delta_k$ is nef, then we already obtain a minimal model.
		Otherwise, by Proposition \ref{prop: cone theorem for semi-stable fourfold over Spec R} there exists a $(K_{X_k} + \Delta_k)$-negative extremal ray spanned by a curve $\Sigma \subseteq X_{k,s}$ such that $(K_{X_k} + \Delta_k + \lambda_k A_k) \cdot \Sigma = 0$.
		Let $L_k$ be a nef $\bQ$-divisor such that $L_k^\perp = \bR[\Sigma]$.
		Then possibly replacing $L_k$ by a sufficiently large multiple, $G_k := L_k - (K_{X_k} + \Delta_k)$ has positive intersection with every one-cycle in $\overline{\mathrm{NE}}(X) \setminus \{0\}$, and hence it is ample by Kleiman's ampleness criterion (\cite[Theorem 1.18]{KM98}).
		Therefore by Proposition \ref{prop: L semiample mixed char} $L$ is semiample and defines a contraction $f: X_k \to Z$. If $\dim Z < \dim X_k$, then we stop (and obtain a Mori fiber space). If $f: X_k \to Z$ is a divisorial contraction, then we set $X_{k+1} := Z$ and continue. If $f: X_k \to Z$ is a flipping contraction, then the flip $f^+: X^+_k \to Z$ exists by Theorem \ref{thm: existence of flips for semi-stable fourfold over curve with general coef}, and hence we can set $X_{k+1} := X^+_k$ and continue the MMP.
		
		\medskip
		
		Thus we can run a $(K_X + \Delta)$-MMP with scaling of $A$.
		Since the special fiber is contained in the non-klt locus, there is no infinite sequence of flips by Theorem \ref{thm: special termination mixed}.
		
		If $K_X+\Delta$ is pseudo-effective, then $L_k$ is always big in each step, which implies that the case $\dim Z < \dim X_k$ cannot occur. Hence we can keep on running this MMP until $K_{X_k} + \Delta_k$ is nef, which means that we obtain a minimal model. This proves (1).
		
		If $K_X +\Delta$ is not pseudo-effective, then $K_{X_k} + \Delta_k$ can never be nef. Hence this MMP must terminate with a Mori fiber space $X_k \to Z$, which proves (2).
	\end{proof}
	
	\begin{proof}[Proof of Theorem \ref{thm: main thm char p}]
		Similar to the proof of Theorem \ref{thm: main thm mixed}. The cone theorem and the contraction theorem are established in Proposition \ref{prop: cone theorem for semi-stable fourfold over Spec R} and Proposition \ref{prop: L semiample mixed char}. The existence of flips is shown in Theorem \ref{thm: existence of flips for semi-stable fourfold over curve with general coef} and the termination of flips is ensured by Theorem \ref{thm: special termination char p}.
	\end{proof}
	
	\begin{rem}
		In Theorem \ref{thm: main thm char p}, if $K_X+\Delta$ is not pseudo-effective, we can still run a $(K_X+\Delta)$-MMP with scaling of an ample divisor, and it will terminate with a model $(X_k, \Delta_k)$ where there exists a $(K_{X_k} + \Delta_k)$-negative curve $\Sigma \subseteq X_k$ such that the divisor $L_k$ satisfying $L_k \cdot \Sigma = 0$ is nef but not big. In this case we expect that $L_k$ defines a Mori fiber space, but the semiampleness of $L_k$ is not known since so far we do not have the corresponding base free point theorem for threefolds over imperfect fields.
	\end{rem}
	
	\section{Relative MMP over Dedekind schemes}\label{section: global MMP}
	As applications of Theorem \ref{thm: main thm mixed} and Theorem \ref{thm: main thm char p}, in this section we prove Corollary \ref{cor: log terminal model in a neighborhood of s}, Theorem \ref{thm: MMP of strictly semi-stable fourfolds} and Theorem \ref{thm: MMP for strictly semi-stable fourfolds over C in char>0}, where we replace the DVR $R$ by the Dedekind scheme $V$ in the setting. The key to extend the MMP from a local setting to a global setting is to use the semiampleness of some divisor to extend the contraction to an open neighborhood of a given point.
	
	\begin{proof}[Proof of Corollary \ref{cor: log terminal model in a neighborhood of s}]
		We will run a $(K_X+\Delta)$-MMP with scaling of an ample divisor $A$ over a neighborhood of $s\in V$. We will repeatedly replace $V$ by some open neighborhood $s\in U\subseteq V$ and $X$ by $X\times_{V}U$. So we may assume that the residue fields of $V$ do not have characteristic 2, 3 or 5. Let $R := \cO_{V, s}$ and $\mathcal{X}_s:=X\times_V\Spec R$. Let $\eta$ be the generic point of $R$.
		
		Let Cartier divisors $D_1,\cdots,D_m$ be generators of $\mathrm{N}^1(X/V)$. Then by considering the closure of a divisor, we see that the restriction maps $$\mathrm{N}^1(X/V)\to\mathrm{N}^1(\Xx_s/\Spec R),~\mathrm{N}^1(X/V)\to\mathrm{N}^1(X_U/U)$$
		are surjective since $X$ is $\Qq$-factorial. In particular, an ample divisor on $\Xx_s$ can be lifted to an ample divisor on $X_U$.
		
		If $K_X+\Delta$ is nef over $s$, then it is nef over $\eta$. This is an easy consequence of the well-known fact that ampleness is an open condition. Since $(X_\eta,\Delta_\eta)$ is a three-dimensional dlt pair defined over a field of characteristic 0, $K_{X_\eta}+\Delta_\eta$ is semiample by the abundance theorem. By Lemma \ref{lem: semiample is an open condition}, $K_X+\Delta$ is semiample, thus nef, over a neighborhood of $\eta$. Since $V$ is an excellent Dedekind scheme, we can then find an open neighborhood $s\in U\subseteq V$ such that $K_X+\Delta$ is nef over $U$. Hence $(X, \Delta)$ itself is a log terminal model over $U$, and we are done.
		
		If $K_X+\Delta$ is not nef over $s$, then there exists a $(K_X+\Delta)|_{\mathcal{X}_s}$-negative extremal ray on $\Xx_s$ which is spanned by a curve $\Sigma \subseteq \Xx_s$. As in the proof of Theorem \ref{thm: main thm mixed}, we can find a $\Qq$-divisor $L$ on $X$ such that $L|_{\Xx_s}$ is nef with $(L|_{\Xx_s})^{\bot}=\Rr[\Sigma]$ and $G:=L-(K_X+\Delta)$ is ample over $s$.
		Possibly shrinking $V$ we may assume that $G$ is ample over $V$.
		By Proposition \ref{prop: L semiample mixed char} and Lemma \ref{lem: semiample is an open condition}, $L$ is semiample over $V$ and defines a contraction $f:X\to W$. Without loss of generality, we may assume that $D_1\cdot\Sigma>0$. Then by Proposition \ref{prop: L semiample mixed char} again there exist $a_2,\cdots,a_m\in\Qq$ such that $$D_i-a_iD_1\sim_{\Qq,\mathcal{W}_s}0$$ 
		for $2\le i\le m$, where $\mathcal{W}_s=W\times_{V}\Spec R$.
		Possibly shrinking $V$ we may assume that $D_i-a_iD_1\sim_{\Qq,W}0$. Thus the relative Picard number $\rho(X/W)=1$. It remains to prove that any $f$-numerically trivial $\Qq$-Cartier divisor $D$ descends to $W$.
		By Lemma \ref{lem: semi-stable in a neighborhood in mixed char}, possibly shrinking $V$ we may assume that $(X,\Delta+X_t)$ is dlt for any $t\in V\setminus\{s\}$. Thus we can apply Proposition \ref{prop: L semiample mixed char} and Lemma \ref{lem: semiample is an open condition} to $\Xx_t$ for any $t\in V$. In particular, $L$ and $L+D$ define the same contraction over a neighborhood of $t$. Hence $D$ descends to $W$ over a neighborhood of $t$. Therefore $D$ descends to $W$ over $V$.
		
		\smallskip
		
		If $f$ is not birational, then $f:X\to W$ is a Mori fiber space. Thus we may assume that $f$ is a birational contraction.
		
		If $f$ is a divisorial contraction, then we set $X_1=W$ and continue the MMP.
		If $f$ is a flipping contraction, then the flip $f^+: X^+\to W$ exists by the following claim, and we set $X_1=X^+$ and continue the MMP.
		\begin{claim}\label{claim: flip can extend to an open neighborhood}
			If $f$ is a flipping contraction, then $f_{t}:\Xx_t\to\mathcal{W}_t$ (resp. $f_U:X_U\to W_U$) is either a flipping contraction or an isomorphism for any $t\in V$(resp. any open subset $U$ of $V$). Furthermore, the flip $f^+: X^+\to W$ exists.
		\end{claim}
		\begin{proof}
			The restriction map $\mathrm{N}^1(X/W)\to \mathrm{N}^1(\Xx_t/\mathcal{W}_t)$ is surjective, and so $\rho(\Xx_t/\mathcal{W}_t)\le1$.  If $\rho(\Xx_t/\mathcal{W}_t)=0$ then $f_t$ is an isomorphism. If $\rho(\Xx_t/\mathcal{W}_t)=1$ then $f_t$ is a flipping contraction. The same argument holds for $f_U$.
			
			Since the construction of flips is local, it suffices to show that for any $t \in V$ there exists a neighborhood $U_t \ni t$ such that $\mathrm{Proj}_W \bigoplus_{m\ge0} f_\ast \cO_X(m(K_X + \Delta))$ is finitely generated over $U_t$.
			If $f_t$ is an isomorphism, then it is obvious. If $f_t$ is a flipping contraction, then by Theorem \ref{thm: existence of flips for semi-stable fourfold over curve with general coef} the flip $(f_t)^+$ exists.
			Taking the closure of $(f_t)^+$ we get a projective morphism $X'\to W$. Then there is an open neighborhood $U_t$ of $t$ such that $X'_{U_t}\to W_{U_t}$ is small and the birational transform of $K_X+\Delta$ is ample over $W_{U_t}$. Therefore $(f_{U_t})^+:X'_{U_t}\to W_{U_t}$ is the flip of $f_{U_t}$, and hence the desired finite generation follows.
		\end{proof}

		Finally this MMP terminates by the same reason as in the proof of Theorem \ref{thm: main thm mixed}. In particular, we get a log terminal model over a neighborhood of $s$ if $(K_X+\Delta)|_{\Xx_s}$ is pseudo-effective, or a Mori fiber space over a neighborhood of $s$ if $(K_X+\Delta)|_{\Xx_s}$ is not pseudo-effective.
	\end{proof}
	
	Furthermore, if $(X, \Delta) \to V$ is a dlt morphism, then we can run an MMP not only over an open neighborhood but globally over the base.
	
	\begin{prop}\label{prop: MMP of pair such that (X,X_s) is dlt}
		Let $V$ be an excellent Dedekind scheme whose residue fields are perfect and do not have characteristic 2, 3 or 5. Let $(X,\Delta)$ be a four-dimensional $\Qq$-factorial pair projective over $V$ such that $X_\bQ \neq \emptyset$. Assume that $(X,\Delta+X_t)$ is dlt for any closed point $t\in V$, where $X_t$ is the fiber of the natural morphism $\phi:X\to V$. Then we can run a $(K_X+\Delta)$-MMP with scaling of an ample divisor over $V$, and
		\begin{enumerate}
			\item if $K_X+\Delta$ is pseudo-effective over $V$, then this MMP terminates with a minimal model;
			\item if $K_X+\Delta$ is not pseudo-effective over $V$, then this MMP terminates with a Mori fiber space.
		\end{enumerate}
	\end{prop}
	\begin{proof}
		We shall construct a special $(K_X+\Delta)$-MMP with scaling such that the scaling number will decrease after finitely many steps, which is important for us to deduce termination of the MMP.
		
		\medskip
		
		The assumptions in the proposition are preserved under the $(K_X+\Delta)$-MMP over $V$. So it suffices to prove that we can run such an MMP and it terminates. We will use the same notation as in the proof of Corollary \ref{cor: log terminal model in a neighborhood of s}.

		Let $A$ be a $\Qq$-divisor such that for any closed point $t\in V$, there exists $A_t\sim_{\Qq,V}A$ such that 
		\begin{itemize}
			\item $K_X+\Delta+A$ is nef over $V$,
			\item $(X,\Delta+X_t+A_t)$ is dlt, and
			\item $\mathbf{B}_+(A/V)$ does not contain any non-klt centers of $(X,\Delta+X_t+A_t)$.
		\end{itemize}
		To start with, we can choose $A$ to be a small multiple of a sufficiently ample divisor over $V$.
		
		\begin{claim}\label{claim: nef threshold is rational}
			$c:=\inf\{0\le b\in\Rr\mid K_X+\Delta+bA\text{ is nef over $V$}\}$ is a rational number. If $c>0$, then there exists a $(K_X+\Delta)$-negative curve $\Sigma\subseteq X$ over $V$ such that $(K_X+\Delta+cA)\cdot\Sigma=0$.
		\end{claim}
		\begin{proof}
			Let $a:=\inf\{0\le b\in\Rr\mid K_X+\Delta+bA\text{ is nef over } \eta\}$. If $a=c$ then we are done by the cone theorem on $X_\eta$. Otherwise, by the cone theorem and the base point free theorem on $X_\eta$, we see that $a$ is a rational number and $K_X+\Delta+aA$ is semiample over $\eta$. Lemma \ref{lem: semiample is an open condition} implies that $K_X+\Delta+aA$ is actually semiample over an open neighborhood of $\eta$. Thus the $(K_X+\Delta+aA)$-negative curves are supported on only finitely many closed fibers. Notice that any $(K_X+\Delta+aA)$-negative extremal ray is also $(K_X+\Delta)$-negative since $K_X+\Delta+A$ is nef. Then the claim follows by considering adjunction and the cone theorem on each component of those closed fibers.  
		\end{proof}
		
		First we show that we can run a $(K_X+\Delta)$-MMP with scaling of $A$.
		
		\begin{lem}\label{lem: contractions and flips exist}
			Under the same assumptions as in Proposition \ref{prop: MMP of pair such that (X,X_s) is dlt}, let $A$ be a divisor as above. Suppose $c:=\inf\{0\le b\in\Rr\mid K_X+\Delta+bA\text{ is nef over $V$}\}>0$. Then there exist a $(K_X+\Delta)$-negative extremal curve $\Sigma$ over $V$ such that $(K_X+\Delta+cA)\cdot\Sigma=0$, and a contraction $g: X\to Z$ of $\Sigma$ such that $\rho(X/Z)=1$.\par
			Furthermore, if $g$ is a flipping contraction, then the flip $g^+: X^+\to Z$ exists.
		\end{lem}
		
		\begin{proof}
			Now $K_X+\Delta+cA$ is semiample over $V$ by Proposition \ref{prop: L semiample mixed char} and Lemma \ref{lem: semiample is an open condition}. Let $f:X\to W$ be the contraction defined by $K_X+\Delta+cA$.
			By Claim \ref{claim: nef threshold is rational}, $f$ is not an isomorphism.
			
			If $K_X+\Delta$ is nef over $\eta$, then it is semiample and hence nef over an open neighborhood of $\eta$ by the abundance in characteristic 0 and Lemma \ref{lem: semiample is an open condition}.
			Therefore the $(K_X+\Delta)$-negative curves are supported on finitely many closed fibers $X_{t_i}$.
			Considering adjunction and the cone theorem on each component of $X_{t_i}$, we can deduce that there are finitely many curves $C_1, \cdots, C_m$, such that
			$$
			\overline{\mathrm{NE}}(X/W) = \overline{\mathrm{NE}}(X/W)_{K_X+\Delta \ge 0} + \sum_{j=1}^m \bR_{\ge 0}[C_j].
			$$
			Then there exists an ample divisor $H$ such that $(K_X + \Delta +H)^\perp = \bR[C_j]$ for some $j$. Hence by Proposition \ref{prop: L semiample mixed char} and Lemma \ref{lem: semiample is an open condition}, $K_X + \Delta +H$ is semiample over $V$ and defines a desired contraction.
			
			Next assume that $K_X+\Delta$ is not nef over $\eta$.
			Then there is an ample divisor $G_\eta$ on $X_\eta$ such that $K_{X_\eta} + \Delta_\eta + G_\eta$ is nef and $(K_{X_\eta} + \Delta_\eta + G_\eta)^\perp = \bR[\Sigma]$ for some $(K_{X_\eta} + \Delta_\eta)$-negative curve $\Sigma \subseteq X_\eta$.
			Let $G = \overline{G}_\eta$. Note that $G$ is ample over a neighborhood of $\eta$ but may behave badly outside.
			Let $G' := \epsilon G + (1-\epsilon)cA$ for some sufficiently small $\epsilon \in \bQ$, such that for any $t\in V$ there exists $G'_t \sim_{\bQ, V} G'$ satisfying
			\begin{itemize}
				\item $(X,\Delta+X_t+G'_t)$ is dlt, and
				\item $\mathbf{B}_+(G'/V)$ does not contain any non-klt centers of $(X,\Delta+X_t+G'_t)$.
			\end{itemize}
			Since $L' := K_X+\Delta +G'$ is nef over $\eta$, it is semiample and hence nef over an open neighborhood of $\eta$ by the base point free theorem in characteristic 0 and Lemma \ref{lem: semiample is an open condition}.
			If $L'$ is not nef over $W$, then the $L'$-negative curves are supported on finitely many closed fibers. Hence we can argue as above and deduce the conclusion.
			If $L'$ is nef over $W$, then by Proposition \ref{prop: L semiample mixed char} and Lemma \ref{lem: semiample is an open condition} it is semiample over $W$ and hence defines a contraction $f': X \to W'$ over $W$.
			By the choice of $G_\eta$, $\rho(X_\eta/W'_\eta) = 1$.
			If $\rho(X/W') = 1$, then $f'$ is the desired contraction.
			Otherwise, by the proof of Corollary \ref{cor: log terminal model in a neighborhood of s}, $\rho(X_U/W'_U) = 1$ for some open subset $U \subseteq V$.
			Let $\Sigma_1$ be a curve generating $\mathrm{N}_1(X_U/W'_U)$, and let $\Sigma_2 \in \mathrm{N}_1(X/W')$ be a curve which is numerically linearly independent with $\Sigma_1$.
			Then we can find a divisor $H'$ such that $H' \cdot \Sigma_1 = 0$ while $H' \cdot \Sigma_2 < 0$.
			Consider $L'' := K_X + \Delta + G' + \delta H'$ for $\delta \in \bQ$ sufficiently small.
			Then $L''$ is nef over $\eta$ but not nef over $V$.
			Therefore the $L''$-negative curves are supported on finitely many closed fibers. Hence we can argue as above and deduce the conclusion.

			If $g$ is a flipping contraction, then we can apply Claim \ref{claim: flip can extend to an open neighborhood} to conclude the existence of the flip.
		\end{proof}
		
		Then we show the termination of the MMP in a special case.
		
		\begin{lem}\label{lem: MMP iso on open subset terminates}
			Under the same assumptions as in Proposition \ref{prop: MMP of pair such that (X,X_s) is dlt}, if
			$$
			X=X_0\dashrightarrow X_1\dashrightarrow\cdots\dashrightarrow X_i\dashrightarrow\cdots
			$$
			is a sequence of steps of $(K_X+\Delta)$-MMP with scaling of an ample divisor over $V$ such that it is an isomorphism on $X_U$ for some open subset $U \subseteq V$, then this MMP terminates.
		\end{lem}
		
		\begin{proof}
			First we claim that the restriction of the above MMP to $\mathcal{X}_t$ is again an MMP, where $t \in V$ is a closed point. Consider one step of the above MMP $X_i \dashrightarrow X_{i+1}$, and let $f_i: X_i \to Z_i$ be the corresponding contraction.
			Let $(\mathcal{X}_i)_t := X_i\times_V\Spec \cO_{V,t}$ and $(\mathcal{Z}_i)_t := Z_i\times_V\Spec \cO_{V,t}$.
			If $f_i$ is a divisorial contraction, then $f_i|_{(\mathcal{X}_i)_t}$ is either a divisorial contraction or an isomorphism, since $\rho((\mathcal{X}_i)_t / (\mathcal{Z}_i)_t) \le \rho(X_i/Z_i) = 1$. If $f_i$ is a flipping contraction, then by Claim \ref{claim: flip can extend to an open neighborhood} $f_i|_{(\mathcal{X}_i)_t}$ is either a flipping contraction or an isomorphism, and in the former case the flip exists and is the same as the restriction of the flip. Therefore we can restrict the above MMP and obtain an MMP of $\mathcal{X}_t$.
			
			Let $t_1, \cdots, t_l$ the closed points which are not in $U$. By Theorem \ref{thm: main thm mixed}, for each $t_i$ there exists an integer $N_i$ such that $\mathcal{X}_{t_i, j} \dashrightarrow \mathcal{X}_{t_i, j+1}$ are isomorphisms for $j \ge N_i$. In particular, $X_j \dashrightarrow X_{j+1}$ are isomorphisms over $U \cup \{t_i\}$ for $j \ge N_i$. Therefore, $X_j \dashrightarrow X_{j+1}$ are isomorphisms for $j \ge \max\{N_1, \cdots, N_l\}$. Hence the MMP in the lemma terminates.
		\end{proof}
		
		Now we start to construct the special $(K_X+\Delta)$-MMP with scaling of $A$ as we mentioned in the beginning of the proof.
		
		Note that $K_X+\Delta+cA$ is nef over $V$ since $c$ is the nef threshold. If $c=0$, then we already get a log terminal model. If $c>0$, then by Proposition \ref{prop: L semiample mixed char} and Lemma \ref{lem: semiample is an open condition}, $K_X+\Delta+cA$ is actually semiample over $V$ and defines a contraction $f:X\to W$.
		\begin{claim}\label{claim: K_X+Delta+c/2A MMP}
			We can run a $(K_X+\Delta+\frac{c}{2}A)$-MMP over $W$ such that it terminates.
		\end{claim}
		Notice that such an MMP is also a $(K_X+\Delta)$-MMP over $V$ with scaling of $A$.
		\begin{proof}
			If $K_X+\Delta+\frac{c}{2}A$ is nef over $W$, then we already get a log terminal model. Otherwise, first consider the case where $K_X+\Delta+\frac{c}{2}A$ is nef over $W_\eta$.
			Then by the base point free theorem in characteristic 0 and Lemma \ref{lem: semiample is an open condition}, it is semiample over $W_U$ for some open subset $U\subseteq V$.
			By Lemma \ref{lem: contractions and flips exist}, we can run a $(K_X+\Delta+\frac{c}{2}A)$-MMP with scaling with of $\frac{c}{2}A$.
			Note that this MMP is an isomorphism on $X_U$.
			Therefore by Lemma \ref{lem: MMP iso on open subset terminates} this MMP terminates.
			
			Next consider the case where $K_X+\Delta+\frac{c}{2}A$ is not nef over $W_\eta$.
			Then there exist a $(K_{X_\eta}+\Delta_\eta+\frac{c}{2}A_\eta)$-negative extremal curve $\Sigma$ on $X_\eta$, and an effective $\Qq$-divisor $G$ on $X$ such that
			\begin{itemize}
				\item $G$ is ample over $\eta$,
				\item $L=K_X+\Delta+G$ is nef over $\eta$, and
				\item $(L_\eta)^\bot=\Rr[\Sigma]$.
			\end{itemize}
			Possibly replacing $G$ by $(1-\epsilon)cA+\epsilon G$ for $0<\epsilon\ll 1$ and $\epsilon\in\Qq$, we may assume that for any $t\in V$, there exists $G_t\sim_{\Qq,V}G$ such that $(X,\Delta+X_t+G_t)$ is dlt and $\mathbf{B}_+(G/V)$ does not contain any non-klt centers of $(X,\Delta+X_t+G_t)$.
			Just as in the first case, we can run a $(K_X+\Delta+G)$-MMP with scaling of an ample divisor over $W$ and it terminates.
			If the MMP terminates with a Mori fiber space, then we are done.
			Otherwise we obtain a minimal model over $W$.
			So we may assume that $K_X+\Delta+G$ is nef over $W$.
			Then by Proposition \ref{prop: L semiample mixed char} and Lemma \ref{lem: semiample is an open condition} it is semiample over $W$ and defines a contraction $g:X\to Z$. By the proof of Corollary \ref{cor: log terminal model in a neighborhood of s}, there exists an open subset $U$ of $V$, such that $\rho(X_U/Z_U)=1$.
			By Lemma \ref{lem: contractions and flips exist} we can run a $(K_X+\Delta+\frac{c}{2}A)$-MMP over $Z$.
			If each step of the MMP is an isomorphism over $\eta$, then it is an isomorphism on $X_U$, since $\rho(X_U/Z_U)=1$ implies that any contraction contracting a curve in $X_U$ also contracts some curve in $X_\eta$.
			In this situation by Lemma \ref{lem: MMP iso on open subset terminates} the MMP terminates, which is a contradiction since $K_X+\Delta+\frac{c}{2}A$ is not nef over $\eta$.
			Thus we can choose $X_i\dashrightarrow X_{i+1}$ to be the first step in this MMP which is not an isomorphism over $\eta$.
			Replace $X\to W$ by $X_{i+1}\to W$ and continue the discussion as above.
			Then we obtain a sequence of steps of $(K_X+\Delta+\frac{c}{2}A)$-MMP, which must terminate since otherwise restricting to the generic fiber $X_\eta$ we will get an infinite sequence of flips on a dlt threefold in characteristic 0.
		\end{proof}
		
		We run the MMP as in Claim \ref{claim: K_X+Delta+c/2A MMP}.
		If the MMP terminates with a Mori fiber space, then we are done.
		So we may assume that we obtain a model $(X',\Delta')$ such that $K_{X'}+\Delta'$ is nef over $W$. This is also a $(K_X+\Delta)$-MMP with scaling of $A$, and next we prove that the scaling number decreases.
		More explicitly, we claim that
		$$
		c':=\inf\{0\le b\in\Rr\mid K_{X'}+\Delta'+bA'\text{ is nef over $V$}\}<c.
		$$
		Indeed, by the construction above there is no $(K_{X'}+\Delta')$-negative and $(K_{X'}+\Delta'+cA')$-trivial curve, and Claim \ref{claim: nef threshold is rational} implies that $c$ cannot be the nef threshold on $X'$.
		
		
		Let $X=X_0\dashrightarrow X_1\dashrightarrow\cdots\dashrightarrow X_i\dashrightarrow\cdots$ be the sequence of steps of $(K_X+\Delta)$-MMP with scaling of $A$ we constructed above and $\lambda_i$ the scaling numbers. 
		\begin{claim}\label{claim: MMP will terminate}
			The MMP above terminates. 
		\end{claim}
		\begin{proof}
			Assume the opposite. Let $\lim_{i \to \infty}\lambda_i=\lambda$. Then $\lambda_i>\lambda$ since $\lambda_i$ will decrease after finitely many steps. For any $s\in V$, there exists a positive integer $N_s$ such that $X_j\dashrightarrow X_{j+1}$ are isomorphisms on $\Xx_s$ for $j\ge N_s$. When $s$ is a closed point this follows from Theorem \ref{thm: main thm mixed}, and when $s=\eta$ this follows from the MMP for threefolds in characteristic 0. In particular, $X_j\dashrightarrow X_{j+1}$ are isomorphisms over an open neighborhood of $s$.\par
			First let $s=\eta$. Then $K_{X_j}+\Delta_j+\lambda_jA_j$ are nef on $X_\eta$ for $j\ge N_\eta$. Let $a:=\inf\{0\le b\in\Rr\mid K_{X_{N_\eta}}+\Delta_{N_\eta}+bA_{N_\eta}\text{ is nef over } \eta\}$. Then we can see that $a$ is a rational number and $a\le\lambda$.
			By the base point free theorem in characteristic 0 and Lemma \ref{lem: semiample is an open condition}, $K_{X_{N_\eta}}+\Delta_{N_\eta}+aA_{N_\eta}$ is semiample and hence nef over an open subset $U$ of $V$. Note that the MMP $X_j\dashrightarrow X_{j+1}$ only contracts $(K_{X_j}+\Delta_j)$-negative and $(K_{X_j}+\Delta_j+\lambda_jA_j)$-trivial extremal rays. Since $\lambda_j>a$ for any $j>0$, $X_j\dashrightarrow X_{j+1}$ is an isomorphism on $\Xx_s$ for any $s\in U$ and $j\ge N_\eta$. Let $\{s_1,...,s_l\}=V\setminus U$. Then $X_j\dashrightarrow X_{j+1}$ will be isomorphisms when $j>\max\{N_\eta,N_{s_1},...,N_{s_l}\}$, which is a contradiction since we assume the MMP does not terminate.
		\end{proof}
		
		We have constructed a $(K_X+\Delta)$-MMP with scaling of an ample divisor over $V$ and proved that the MMP terminates. Hence the proposition follows.
	\end{proof}

	\begin{proof}[Proof of Theorem \ref{thm: MMP of strictly semi-stable fourfolds}]
		By the definition of strict semi-stability, we can see that $X$ satisfies the assumptions in Proposition \ref{prop: MMP of pair such that (X,X_s) is dlt}. 
		Therefore the statement follows.
	\end{proof}

	Note that the proof of Proposition \ref{prop: MMP of pair such that (X,X_s) is dlt} also applies to the positive characteristic case when $K_X+\Delta$ is big.
	
	\begin{prop}\label{prop: MMP for big pair over C in char>0}
		Let $(X,\Delta)$ be a four-dimensional $\Qq$-factorial klt pair projective and surjective over a curve $C$ which is defined over a perfect field of characteristic $p>5$. Assume that $(X,\Delta+X_t)$ is dlt for any closed point $t\in C$, where $X_t$ is the fiber of the natural morphism $\phi:X\to C$.\par
		If $K_X+\Delta$ is big over $C$, then we can run an $(K_X+\Delta)$-MMP with scaling of an ample divisor over $C$ which terminates with a good minimal model.
	\end{prop}
	\begin{proof}
		Since $K_X+\Delta$ is big, the semiampleness results needed in the proof of Proposition \ref{prop: MMP of pair such that (X,X_s) is dlt} hold by Proposition \ref{prop: L semiample mixed char}(2) and \cite[Theorem 1.4]{DW19}. Note that the arguments in the proof of Corollary \ref{cor: log terminal model in a neighborhood of s} also hold in positive characteristic under this stronger condition. Thus we can argue as in the proof of Proposition \ref{prop: MMP of pair such that (X,X_s) is dlt} and get a log terminal model. It is actually a good minimal model by \cite[Proposition 5.1]{HW20} and Lemma \ref{lem: semiample is an open condition}.
	\end{proof}
	
	\begin{proof}[Proof of Theorem \ref{thm: MMP for strictly semi-stable fourfolds over C in char>0}]
		By the definition of strict semi-stability, we can see that $X$ satisfies the assumptions in Proposition \ref{prop: MMP for big pair over C in char>0}.
		Therefore the statement follows.
	\end{proof}

\end{document}